\documentclass{amsart}
\usepackage[utf8]{inputenc}

\usepackage{amsmath,amsfonts, amssymb}
\usepackage{graphicx}
\usepackage{xcolor}
\usepackage{enumerate}
\usepackage{enumitem}
\usepackage{cite}
\usepackage{tabu}

\usepackage{stmaryrd} %for blackboard bold brackets

\usepackage{fullpage}
\usepackage[active]{srcltx}

\usepackage{changepage}
\usepackage{chngcntr}

\usepackage{multibib}

\usepackage{url}
\usepackage{hyperref}
\hypersetup{
	colorlinks,
	linkcolor={red!50!black},
	citecolor={blue!50!black},
	urlcolor={blue!80!black}
}  %%% This tones down the bright green hyperlink box and its other colors...

\theoremstyle{plain}
\newtheorem{theorem}{Theorem}[section]

\newtheorem{cor}[theorem]{Corollary}

\newtheorem{prop}[theorem]{Proposition}
\newtheorem{lemma}[theorem]{Lemma}
\newtheorem{claim}[theorem]{Claim}

\newtheorem*{proptanglesum}{Proposition~\ref{prop:unknottingtanglesum}}

\newtheorem*{properrationalunknottingtanglesum}{Proposition~\ref{prop:properrationalunknottingtanglesum}}

\theoremstyle{definition}
\newtheorem{defn}[theorem]{Definition}
\newtheorem{remark}[theorem]{Remark}

\newcommand{\comment}[1]{}

\newcommand{\bdry}{\ensuremath{\partial}}

\newcommand{\boundary}{\bdry}

\newcommand{\uprop}{\ensuremath{u_{\rm{prop}}}}

\newcommand{\nbhd}{\ensuremath{\mathcal{N}}}

\newcommand{\Q}{\ensuremath{\mathbb{Q}}}

\newcommand{\Z}{\ensuremath{\mathbb{Z}}}

\newcommand{\cut}{\ensuremath{\backslash}}

\newcommand{\thetan}{\theta}

\newcommand{\emphsz}[1]{\textbf{#1}}
\newcommand{\trivtheta}{\theta_0}

\definecolor{amaranth}{rgb}{0.9, 0.17, 0.31} %dark red
\definecolor{carrotorange}{rgb}{0.93, 0.57, 0.13} %orange
\definecolor{citrine}{rgb}{0.89, 0.82, 0.04} %dark yellow
\definecolor{dartmouthgreen}{rgb}{0.05, 0.5, 0.06} %green
\definecolor{ballblue}{rgb}{0.13, 0.67, 0.8} %blue
\definecolor{ceruleanblue}{rgb}{0.16, 0.32, 0.75} %deeper blue
\definecolor{amethyst}{rgb}{0.6, 0.4, 0.8} %purple
\definecolor{amber}{rgb}{1.0, 0.75, 0.0} %amber
\definecolor{burlywood}{rgb}{0.87, 0.72, 0.53} %beigebrown
\title{Primality of theta-curves with proper rational tangle unknotting number one}

\makeatletter
\@namedef{subjclassname@1991}{2020 Mathematics Subject Classification}
\makeatother

\author{Kenneth L. Baker}
\address{Department of Mathematics, 
University of Miami}   
\email{k.baker@math.miami.edu}                      

\author{Dorothy Buck}
\address{Department of Mathematics \& Department of Biology, 
Duke University}   
\email{dbuck@math.duke.edu}        

\author{Danielle O'Donnol}
\address{Department of Mathematics,
Marymount University}   
\email{dodonnol@marymount.edu} 

\author{Allison H. Moore}
\address{Department of Mathematics \& Applied Mathematics,  Virginia Commonwealth University}   
\email{moorea14@vcu.edu}   

\author{Scott Taylor}
\address{Department of Mathematics, 
Colby College}   
\email{scott.taylor@colby.edu}



\date{\today}
\begin{document}

\begin{abstract}
    We show that if a composite $\theta$-curve has (proper rational) unknotting number one, then it is the order 2 sum of a (proper rational) unknotting number one knot and a trivial $\theta$-curve. We also prove similar results for 2-strand tangles and knotoids.
\end{abstract}

\maketitle

\tableofcontents   % Remove when done?

\section{Introduction}
A spatial graph is a tame embedding of an abstract finite graph in the three-sphere $S^3$.  A $\theta$-curve is a spatial graph consisting of two vertices each joined by three edges. A spatial, planar graph is considered \emphsz{ unknotted} if it is contained in an embedded two-sphere, and the \emphsz{ unknotting number} $u$ of a spatial graph is the minimum number of crossing changes required to make it unknotted.  As with knots, the unknotting number is an extremely basic invariant of abstractly planar spatial graphs whose calculation can quickly become intractable. To illustrate the difficulty in understanding unknotting number, note that the additivity of unknotting number under connected sum is open for both knots and spatial graphs. However, Scharlemann showed that all unknotting number one knots are prime \cite{Scharlemann:Unknotting}.  We prove a version of Scharlemann's theorem for $\theta$-curves. As we explain in Section \ref{bio}, we were motivated by certain biological considerations. 

Although, initially one might expect the version for $\theta$-curves to follow easily from Scharlemann's theorem or the machinery of double-branched covers, the situation is subtler. In particular, for knots and $\theta$-curves, there are two natural notions of connected sum: order 2 connected sum and order 3 connected sum, both depicted in Figure~\ref{fig:thetasums}.  In particular, observe that a composite $\theta$-curve which is the order 2 connected sum of a trivial $\theta$-curve with an unknotting number one knot, itself has unknotting number one. It turns out, however, that is the only possible exception to the claim that unknotting number one $\theta$-curves are prime. We prove:

\begin{theorem}\label{main theorem}
Let $\thetan$ be a $\theta$--curve.  If $u(\thetan)=1$ then either
\begin{itemize}
    \item $\thetan$ is prime with respect to both order 2 connected sum $\#_2$ and order 3 connected sum $\#_3$, or
    \item $\theta = \trivtheta \#_2 K$ where $\trivtheta$ is the trivial $\thetan$--curve and $K$ is a prime knot with $u(K)=1$. 
    \end{itemize}
In the latter case, up to isotopy, the unknotting crossing change occurs on $K$.
\end{theorem}

\begin{figure}
    \centering
    \includegraphics[width=.9\textwidth]{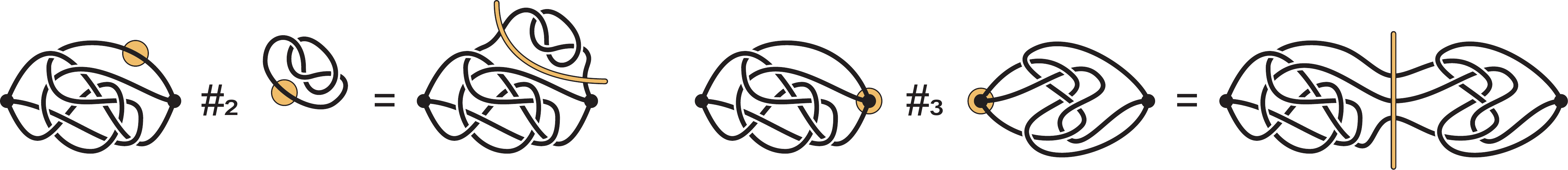}%{thetasums.pdf}
    \caption{Left: an order 2 connected sum of a theta curve and a knot.  Right: an order 3 connected sum of two theta curves. }
    \label{fig:thetasums}
\end{figure}

\begin{figure}
    \centering
    \includegraphics[width=.9\textwidth]{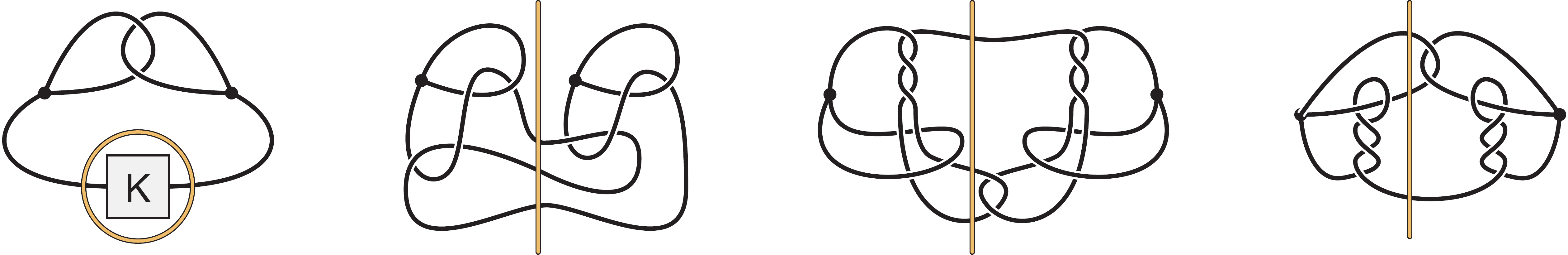}
    \caption{Four examples of composite handcuff curves with unknotting number one. These directly generalize to examples of unknotting proper rational tangle replacements that are not crossing changes.}
    \label{fig:compositehandcuff}
\end{figure}

Theorem \ref{main theorem} follows from a similar theorem for 2-strand tangles. Deferring the definition of prime tangle to Section~\ref{sec: spheres},  we show:

\begin{proptanglesum}%\label{prop:unknottingtanglesum}
Let $T$ be a marked tangle sum of tangles, neither of which is integral. Then $T$ admits no crossing change to an integral tangle.
\end{proptanglesum}

The arguments for Proposition~\ref{prop:unknottingtanglesum} mostly adapt to using the more general proper rational tangle replacements (see Section~\ref{sec:rtr}) in place of crossing changes, though a different result from \cite{GordonLuecke} is employed.

\begin{properrationalunknottingtanglesum}
Let $T$ be a marked tangle sum of tangles, neither of which is integral. Then $T$ admits no proper rational tangle replacement to an integral tangle.
\end{properrationalunknottingtanglesum}

Consequently, we obtain a generalization of Theorem~\ref{main theorem} using proper rational tangle replacements.  The \emphsz{proper rational unknotting number}  $\uprop$ of a spatial graph  is the minimum number of proper rational tangle replacements required to make it unknotted, see Section~\ref{sec:rtr} and \cite{properrationalunknotting}. 

\begin{theorem}\label{thm:propermain}
Let $\thetan$ be a $\theta$--curve.  If $\uprop(\thetan)=1$ then either
\begin{itemize}
    \item $\thetan$ is prime with respect to both order 2 connected sum $\#_2$ and order 3 connected sum $\#_3$, or
    \item $\theta = \trivtheta \#_2 K$ where $\trivtheta$ is the trivial $\thetan$--curve and $K$ is a prime knot with $\uprop(K)=1$. 
    \end{itemize}
In the latter case, up to isotopy, the unknotting proper rational tangle replacement occurs on $K$.
\end{theorem}

Ultimately, Theorem~\ref{main theorem} is a special case of Theorem~\ref{thm:propermain} since crossing changes are proper rational tangle replacements.  All of the arguments involving crossing changes readily adapt to proper rational tangle replacements with the exception of Proposition~\ref{prop:unknottingtanglesum} which needs its generalization Proposition~\ref{prop:properrationalunknottingtanglesum}.

\begin{proof}[Proof of Theorem~\ref{thm:propermain} (and Theorem~\ref{main theorem}).]
Suppose $\thetan$ is a non-prime $\theta$--curve with $\uprop(\thetan) = 1$. 

The argument splits into two cases.

\noindent
{\bf Case 1.} Assume $\thetan$ is locally knotted.   Then Proposition~\ref{prop:unknottinglocalknot} gives the desired conclusion.  It shows that $\thetan$ is the connected sum of a non-trivial knot $K$ with $\uprop(K)=1$ and a trivial $\theta$--curve.  Moreover, any proper rational unknotting arc for $\thetan$ is isotopic into the ball containing the local knot.

\noindent
{\bf Case 2.} Assume $\thetan$ is a non-trivial vertex sum that is not locally knotted.   Then Proposition~\ref{prop:unknottingvertexsum} shows that this contradicts the assumption that $\uprop(\thetan)=1$.
\end{proof}

\begin{remark}\label{rem:handcuff} In the statements of Theorems \ref{main theorem} and \ref{thm:propermain}, the restriction to $\theta$-curves is crucial. Figure~\ref{fig:compositehandcuff} shows several examples of order 2 and order 3 composite \emphsz{handcuff curves} of unknotting number one. 
In the first one, observe that $K$ can be any knot, and so an unknotting number one handcuff curve can have arbitrarily many summands. The second and third are vertex sums of two knotted handcuff curves.  The fourth is a sum of a knotted theta curve and a knotted handcuff curve. Note however, in all these examples of Figure~\ref{fig:compositehandcuff}, there is an unknotting crossing change disjoint from a summing sphere. Must that always be the case?  
\end{remark}

\subsection{Knotoids}

A knotoid (due to Turaev) is an equivalence class of an oriented regular immersion of  an interval lying in a surface $\Sigma$, where $\Sigma$ is typically $S^2$ or $D^2$. The notion of equivalence is induced by planar isotopy and the three Reidemeister moves for knots, fixing the endpoints of the interval. 
%The semigroup of knotoids is isomorphic with the semigroup of oriented (and ordered) $\thetan$-curves containing at least one trivial constituent knot \cite[Theorem 6.2]{Turaev}.
%In particular, from a knotoid $k$ a $\thetan$-curve can be constructed by joining the two endpoints of the knotoid with two arcs above and below the surface $\Sigma$, which together form the trivial constituent. 
%Likewise, a $\thetan$-curve containing a constituent unknot corresponds with a knotoid $k$ after an isotopy of $S^3$ that makes the constituent unknot transverse to $\Sigma$, with the remaining edge projecting to the knotoid diagram $k$ in $\Sigma$. 
An embedding of the interval into $\Sigma$ is a trivial knotoid. 
%any knotoid equivalent to the trivial knotoid via planar isotopy and the Reidemeister moves is considered trivial. 
Analogously to the definition for knots, the unknotting number of a knotoid is the minimum number of crossing changes required to transform $k$ into the trivial knotoid.  We may similarly define the proper rational unknotting number of a knotoid.

Knotoids form a semigroup under the operation of knotoid multiplication, which is equivalent to vertex connected sum (i.e. an order 1 connected sum $k_1 \#_1 k_2$). 
Moreover, an order 2 connected sum $k_1 \#_2 K_2$ is equivalent to the knotoid product $k_1 \#_1 k_2$, where knotoid $k_2$ is obtained from the knot $K_2$ by removing a small open interval. 
A knotoid $k$ is prime if any decomposition $k=k_1 \#_1 k_2$ implies that $k_1$ or $k_2$ is trivial.

By \cite[Theorem 6.2]{Turaev}, there is a one-to-one correspondence between knotoids and oriented, ordered theta curves containing a fixed trivial constituent knot. In particular, given a knotoid $k$, the curve $\theta(k)$ is constructed by joining the two endpoints of the knotoid with two arcs above and below the surface $\Sigma$, which together form a trivial constituent. 
The product $k_1\#_1 k_2$ of knotoids corresponds with the vertex sum $\theta(k_1)\#_3 \theta(k_2)$ of $\theta$-curves.

\begin{cor}
Let $k$ be a knotoid.  If $\uprop(k)=1$, then $k$ is prime.  Hence, if $u(k)=1$, then $k$ is prime.
\end{cor}

\begin{proof}
Theorem~\ref{thm:propermain} implies that either $\thetan(k)$ is prime with respect to order 2 and 3 connected sums or that $\thetan(k) = \theta_0 \#_2 K$ for some prime knot $K$ of proper rational unknotting number one. 
In the the former case, $k$ is prime, and in the latter case, the knotoid $k$ is both prime and of knot-type, meaning the endpoints of $k$ are contained in the same region of $\Sigma$.  The last conclusion follows since a crossing change is a proper rational tangle replacement.
\end{proof}

A \emphsz{forbidden move} in a knotoid diagram is a planar move in which a single strand passes over or under a vertex. See \cite[Figure 1.2]{Barbensi}. Any knotoid can be made trivial via a finite sequence of forbidden moves. Consequently, the \emphsz{f-distance} between knotoids is the (well-defined) minimal number of forbidden moves required to transform one knotoid into another. 
A forbidden move in a knotoid $k$ corresponds directly with a crossing change in the associated $\thetan$-curve $\theta(k)$. Therefore, we obtain the immediate corollary:

\begin{cor}
A knotoid with $f$-distance 1 is prime. \qed
\end{cor}

\subsection{Biological Motivation}\label{bio}
The study of both $\theta$-graphs and knotoids is externally motivated by questions arising in molecular biology. 
Knotoids have recently been utilized to model knotting and entanglement in open proteins \cite{Goundaroulis1, Goundaroulis2}. 
Spatial theta-curves arise as intermediary structures during the process of DNA replication in circular plasmids and in organisms with circular DNA (eg \emphsz{E. coli}). 
Replication is the process through which a single DNA molecule reproduces to form two new identical DNA molecules.  
During this process the DNA double helix is partially unwound in the replication bubble; the two partially replicated strands (leading and lagging) together with the remaining parent DNA form a $\theta$-curve structure. % \cite{LillyCamps}. 
%Our interest in unknotting number and $\thetan$--curves was motivated by an interested in understanding the knotting that occurs during DNA replication. 
Intriguingly, this $\thetan$--curve can be knotted \cite{AdamsCozzarelli, VigueraSchvartzman, SantamariaSchvartzman, LopezSchvartzman}.  
During replication, packing, and other cellular processes, enzymatic complexes (e.g. topoisomerases) mediate strand-passage events that knot and unknot DNA. 
This leads naturally to an investigation of unknotting number in knots and $\thetan$--curves. 
Unknotting number one $\thetan$--curves are precisely the knotted structures that may arise as replication intermediaries following a single strand-passage event.
%The particular knot types that are formed in replicating DNA are a direct result of the spatial formation of the molecule.   
A classification of such structures and understanding of unknotting processes in spatial $\theta$--curves can shed light on both the knotting that occurs during replication, and the processes that drive it.

\section{Definitions and Conventions}
We consider spatial $\thetan$-curves up to ambient isotopy. 
Because $\theta$-curves are trivalent graphs, there is no difference between the notions of equivalence from rigid vertex isotopy or topological (pliant) isotopy; the results stated throughout hold in either setting.

We consider pairs $(M,t)$ where $M$ is a 3-manifold obtained by removing some finite number (possibly zero) of open balls from $S^3$ and $t \subset M$ is a properly embedded 1-dimensional CW-complex. It is a \emphsz{tangle} if at least one ball has been removed from $S^3$ and if  $t$ is a 1--manifold. It is a \emphsz{2-strand tangle} if $t$ is the disjoint union of two arcs and $M$ is a 3-ball. For a surface $S \subset M$, we write $S \subset (M,t)$ to mean that $S$ is transverse to $t$. It is \emphsz{essential} if it not a 2-sphere bounding a 3-ball disjoint from $t$, is incompressible, and is not $\boundary$-parallel in the exterior of $t$. If $S$ is a sphere with $|S \cap t| \leq 3$, it is a \emphsz{summing sphere}. If $S$ is a sphere with $|S \cap t| = 4$, it is a \emphsz{Conway sphere}. 

\subsection{Rational tangles}\label{sec: rt}

%  A two-string tangle decomposition of a link $L$ is a (non-unique) description
% \[
%     (S^3, L) = (B,t) \cup_S (B', t'),
% \]
% where $\partial B =\partial B'$ is a sphere intersecting $L$ transversely in four points, $t = B\cap L$, $t' = B' \cap L'$. The sphere $S=\partial B$ is called a Conway sphere. 

If a two-string tangle $(B,t)$ is equipped with a homeomorphism taking $(B, t)$ to the ball of radius one in $\mathbb{R}^3$ and $\partial t$ is taken to the NE, SE, NW and SW corners of the equatorial sphere at $z=0$, i.e. the set of points $\{(\pm \cos(\pi/4), \pm(\sin(\pi/4), 0)\}$, we say that $(B, t)$ is a \emphsz{marked} two-string tangle. The \emphsz{numerator closure}, or  \emphsz{denominator closure}, of a marked two-string tangle are the knot or 2-component links obtained by joining the northern and southern points, or eastern and western points, respectively, by the equatorial arcs as shown in Figure \ref{fig:marked tangle}, left.
\begin{figure}
    \centering
    \includegraphics[width=.9\textwidth]{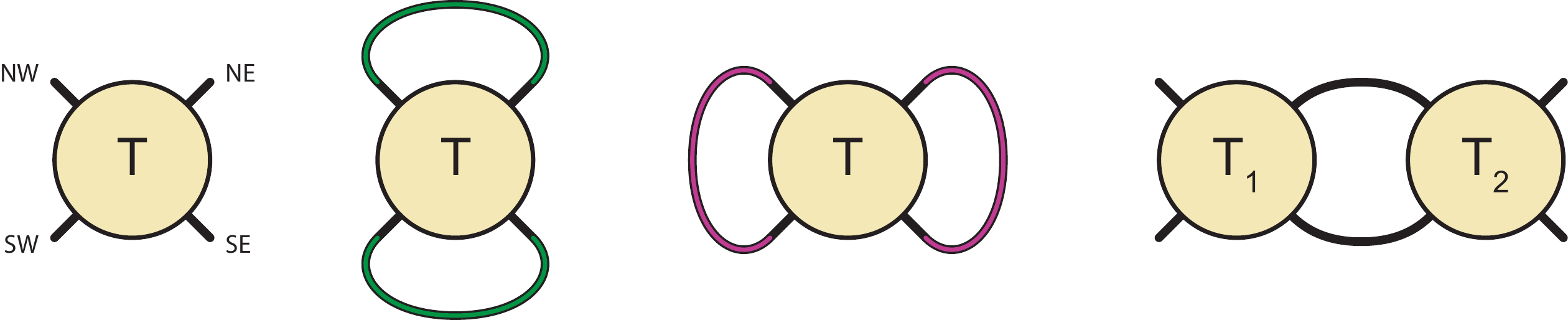}%{marked tangle.png}
    \caption{A marked two-string tangle, numerator closure indicated in green, denominator closure in magenta, and a tangle sum of marked tangles.}
    \label{fig:marked tangle}
\end{figure}
Two marked two-string tangles $T_i = (B_i, t_i)$ for $i = 1,2$ are \emphsz{equivalent} if they are isotopic by an isotopy preserving the markings. 

From two marked tangles $T_1, T_2$, we also may form a new marked two-string tangle $T_1 + T_2$, called the \emphsz{tangle sum}, by identifying the eastern half sphere of $B_1$ relative the points $\{NE, SE\}$ with the western half sphere of $B_2$ relative the points $\{NW, SW\}$, and isotoping $t_1 \cup t_2$ relative to its boundary to be properly embedded in $B_1 \cup B_2$, as in Figure \ref{fig:marked tangle}, right. The tangle sum is \emphsz{essential} if the summing disk (the image of the half spheres that were glued together) is not parallel (via an isotopy transverse to $t_1 \cup t_2$) to the boundary of the 3-ball.

A two-string tangle $(B, t)$ is \emphsz{trivial} if $t$ is isotopic into $\partial B$. It is \emphsz{split} if there is a properly embedded disk separating the strands. A marked two-string trivial tangle is \emphsz{rational}. Rational tangles are split. Famously, Conway \cite{conway} constructed a bijection between rational tangles and $\Q \cup \{\infty\}$ using continued fractions. Any rational number $\frac{\beta}{\alpha}$ can be expressed non-uniquely as a continued fraction
\[
    [a_0, a_1, \cdots, a_n] = a_0 - \frac{1}{a_1 - \frac{1}{a_2 \cdots  - \frac{1}{a_n}}}.
\]
The coefficients of the continued fraction determine the rational tangle. The \emphsz{Conway number} of a rational tangle $T$ is the associated element of $\Q \cup \{\infty\}$. Our conventions (which differ from Conway's) for constructing rational tangles are as given in Figure \ref{fig:rational tangles}
\begin{figure}
    \centering
    \includegraphics[width=0.8\textwidth]{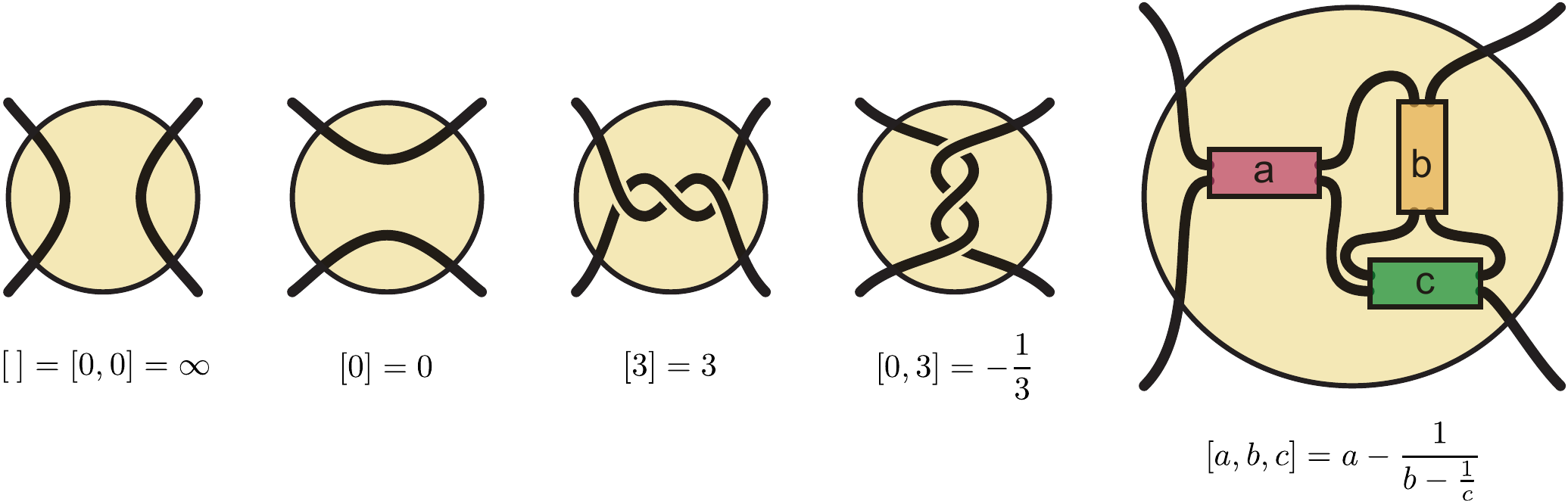}
    \caption{Rational tangle conventions}
    \label{fig:rational tangles}
\end{figure}

Suppose a marked two-string rational tangle $T = T_1 + T_2$ is the Conway sum of marked two-string tangles $T_1$ and $T_2$. Observe that for each $n \in \Z$, we also have $T = (T_1 + [n]) + ([-n] + T_2)$. We will need two decompose two-string tangles as Conway sums and this ambiguity causes some inconvenience. 

\subsection{Crossing Changes and Rational Tangle Replacement}\label{sec:rtr}

Suppose that $M$ is a 3-manifold and that $\tau \subset M$ is a properly embedded graph. Given an arc $\alpha \subset M$ such that $\alpha \cap \tau = \boundary \alpha$ and $\alpha$ is disjoint from the vertices of $\tau$, a \emphsz{rational tangle replacement} (RTR) of $\tau$ with \emphsz{core arc} $\alpha$ is defined as follows.

\begin{defn}\label{defn:CCRTR}
Consider a regular neighborhood $B$ of $\alpha$ and let $\lambda = \tau \cap B$. Then $T = (B, \lambda)$ is a trivial 2-strand tangle. By imposing a marking, without loss of generality, we may take it to be the $\infty$-tangle. (This choice is called a \emphsz{framing} of $\alpha$. Framings are parameterized by the integers.) Let $T' = (B, \lambda')$ be a rational tangle. Then $\tau' = \big((M,\tau) \setminus T\big) \cup T'$ is said to be obtained by a \emphsz{RTR} on $\tau$ with (framed) crossing arc $\alpha$. See Figure~\ref{fig:crossingchange}. When $(B, \lambda')$ is either of the $\pm 1/2$ tangles, the RTR is a \emphsz{crossing change}. When the strands $\lambda'$ are properly homotopic in $B$ to the strands $\lambda$, the RTR is a \emphsz{proper rational tangle replacement}.
\end{defn}

\begin{figure}
    \centering
    \includegraphics[width=0.55\textwidth]{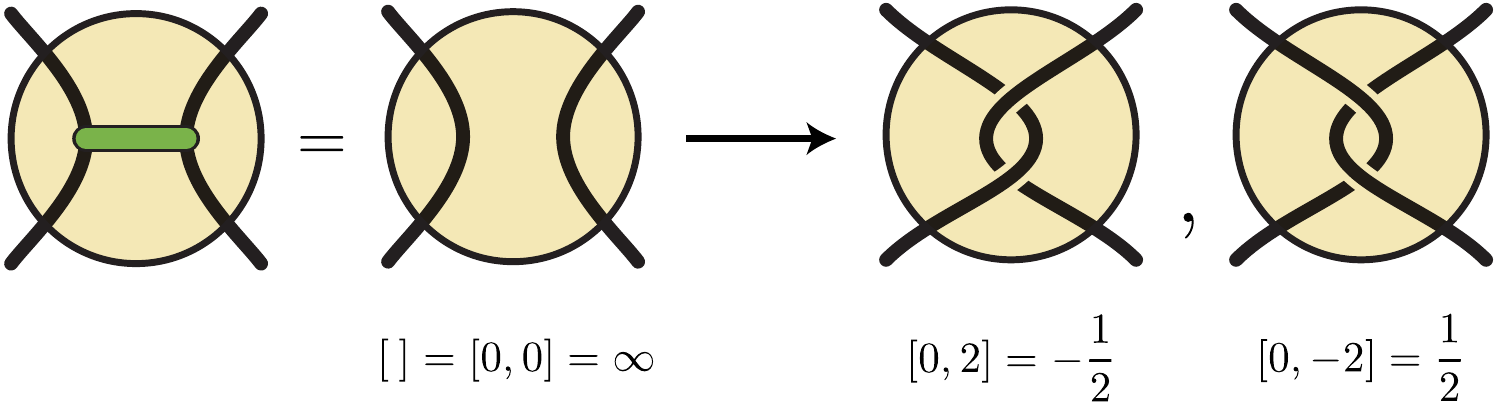}
    \caption{An arc $\alpha$ meeting $\tau$ has a regular neighborhood $B$ meeting $\tau$ in an $\infty$ tangle.   A crossing change of $\tau$ along $\alpha$ replaces this $\infty$ tangle with either of the $\pm \frac12$ tangles.}
    \label{fig:crossingchange}
\end{figure}

\begin{remark}
A crossing change could also be defined via $\pm 1$ Dehn surgery on a loop (called the \emphsz{crossing loop}) bounding a disc intersecting $\tau$ in two points. However, the crossing arc perspective is more useful for us. It is possible to translate results concerning crossing arcs into results concerning crossing loops and vice versa, although care is needed. See \cite{Lackenby}, for example.
\end{remark}

\begin{defn}
Given two rational tangles $(B, \lambda)$ and $(B, \lambda')$ with $\boundary \lambda = \boundary \lambda'$, each contains a properly embedded disc separating the strands, unique up to isotopy disjoint from the strands. The \emphsz{distance} $\Delta$ of an RTR is half the minimal number of intersections between the boundaries of such disks for the tangles $(B, \lambda)$ and $(B, \lambda')$ in Definition \ref{defn:CCRTR}. This coincides with other definitions of the distance between rational tangles in the literature.
\end{defn}

\begin{remark}
A crossing change has $\Delta = 2$. If $\Delta = 1$, the RTR is not a proper RTR.
\end{remark}

\subsection{Sums and Summing Spheres}\label{sec: spheres}
We give a more precise definition of connected sum than what is in the introduction. 

\begin{defn}
Suppose that $K_1$ and $K_2$ are spatial graphs in distinct copies of $S^3$. An order 2 connected sum is obtained as follows. For $i = 1,2$, remove a small open balls from the interior of an edge of $K_1, K_2$, respectively. This results in 3-balls $B_1, B_2$ containing the remnants of $K_1, K_2$ respectively. Glue $\boundary B_1$ to $\boundary B_2$ by a homeomorphism taking $K_1 \cap \boundary B_1$ to $K_2 \cap \boundary B_2$. The result is a spatial graph $K_1 \#_2 K_2$ in $S^3$.  We form an order $k$ connected sum $K_1 \#_3 K_2$ similarly, except we choose each of the small open balls to be centered at a degree $k$ vertex of $K_1, K_2$ respectively \cite{Wolcott}. Observe that the image of $\boundary B_1$ and $\boundary B_2$ is a \emphsz{summing sphere} for the sum. We can also similarly define order 2 sums for tangles. A spatial graph or tangle that has a nontrivial knot as an order 2 connect summand is said to be \emphsz{locally knotted}.

A $\theta$-curve is \emphsz{composite} if it is either an order 2 connected sum of a (possibly trivial) $\theta$-curve and a nontrivial knot or an order 3 connected sum of nontrivial $\theta$-curves. In such a case, as the summing sphere is essential, we also say that the sum is \emphsz{essential}. A $\theta$-curve which is neither trivial nor composite is \emphsz{prime}. 
\end{defn}

Conversely, if $S \subset S^3$ is a sphere transverse to a $\theta$-curve $\thetan$ and intersecting it in two or three points, we can cut $(S^3, \thetan)$ open along $S$ and glue in two copies of $(B^3, \tau)$ where $\tau$ is the cone on $\thetan \cap S$. If $S$ is essential and $|S \cap \thetan| = 3$, then neither component of the result is a trivial $\theta$-curve. If $|S \cap \thetan| = 2$, then one component $(S^3, \theta_1)$ of the result is a $\thetan$-curve and the other $(S^3, K)$ is a knot. 

\begin{remark}
Our definition of prime $\theta$-curve differs slightly from that of Turaev \cite{Turaev} and Matveev-Turaev \cite{MT}. They would say that the order 2 connected sum of a trivial $\theta$-curve with a prime knot is a prime $\theta$-curve, while we say that it is composite. Other than that difference, our definitions coincide.
\end{remark}

\subsection{Preliminary results on Rational Tangles and Conway Sums}

In this section we collect some well-known results which will be useful to us.

\begin{lemma}
\label{lem:lowpunct disks}
Let $T$ be a $2$--strand tangle without local knotting.
Suppose $T$ contains essential disks $D_1$ and $D_2$ each intersecting the strands of $T$ in zero or two points in their interior. 
Then 
\begin{enumerate}
    \item  $\bdry D_1$ and $\bdry D_2$ are isotopic in $\bdry T$,
    \item  the disks have the same number of punctures, and
    \item  after an isotopy of the disks $D_1 \cap D_2 = \emptyset$.
\end{enumerate}
Furthermore, if the disks are disjoint from the strands of $T$, then they are isotopic and $T$ is a rational tangle.
\end{lemma}

\begin{proof}
Out of all such pairs of disks isotopic to $D_1$ and $D_2$, we may assume that $D_1$ and $D_2$ were chosen so as to minimize $|D_1 \cap D_2|$. 

Suppose that there is a circle of intersection in $D_1 \cap D_2$. Let $\zeta$ be an innermost such circle on $D_1$, bounding an innermost disc $E_1 \subset D_1$. Let $E_2 \subset D_2$ be the disc with boundary $\zeta$. 
 Then $E_1 \cup E_2$ is an embedded sphere in $T$.
 Since every sphere in $T$ separates, the parity of the punctures on $E_1$ and $E_2$ is the same. Hence, either they both have a single puncture or they both have $0$ or $2$ punctures. 
 If $E_1$ has $2$ punctures but $E_2$ has $0$, 
    then the sphere $E_1 \cup E_2$ bounds a ball that intersects the strands of $T$ in a single unknotted arc (since $T$ has no local knotting).  This arc may be isotoped rel-$\bdry$ through the ball to an arc in $E_1$ with the same endpoints.  In particular this indicates that there is another disk in this ball disjoint from the unknotted arc and meeting $D_1$ only in $\zeta$.  Hence this other disk is a compressing disk for $D_1$, contradicting that $D_1$ is an essential disk.
If $E_1$ has $0$ punctures but $E_2$ has $2$, then $E_1$ is a compressing disk for $D_2$, contradicting that $D_2$ is an essential disk.
If $E_1$ and $E_2$ both have $0$ punctures or both have $1$ puncture, then there is an isotopy of $E_2$ to $E_1$ through the ball they cobound due to the lack of local knotting and closed components. Pushing $E_2$ slightly past $E_1$ then achieves an isotopy of $D_2$ that reduces $|D_1 \cap D_2|$ contrary to the assumed minimality.
Thus we are left with in the situation that any simple closed curve of the intersection $D_1 \cap D_2$ must bound a disk in each $D_1$ and $D_2$ that is punctured twice by the strands of $T$.  In other words, any simple closed curve of $D_1 \cap D_2$ must bound an unpunctured annulus with $\bdry D_1$ in $D_1$ and another with $\bdry D_2$ in $D_2$.  Let us set this last case aside for the moment.

\medskip

So now assume $\bdry D_1 \cap \bdry D_2 \neq \emptyset$ so that $D_1 \cap D_2$ necessarily contains arcs.
Each arc of $D_1 \cap D_2$ cuts each $D_1$ and $D_2$ into two subdisks. For each $D_1$ and $D_2$, at least one of these two subdisks  is intersected by the strands of $T$ at most once and therefore contains no simple closed curve component of $D_1 \cap D_2$.  
So let $\zeta$ be an arc component of $D_1 \cap D_2$ that is outermost in $D_1$, say, cutting off a disk $E_1$ that intersects $D_2$ only in the arc $\zeta$.
Let $\epsilon$ be the arc $\bdry E_1 \cut \zeta$.  Note that $\epsilon$ lies in $\bdry T$  disjoint from the markings and connecting two points in $\bdry D_2$.  Since $D_2$ is essential in the complement of the strands of $T$, $\bdry D_2$ separates $\bdry T$ into two disks, each with two punctures.  Let $F$ be the component of $\bdry T \cut \bdry D_2$ that contains $\epsilon$. If $\epsilon$ is inessential in $F$, then an isotopy reduces $|\bdry D_1 \cap \bdry D_2|$ taking $\zeta$ to a closed loop $\zeta'$ that bounds a disk $E_1'$ intersecting the strands of $T$ at most once.  Yet as we have seen, such loops imply that either a disk is not essential or there is an isotopy contradicting the minimality of $|D_1 \cap D_2|$. Thus, $\epsilon$ separates the punctures of $F$.  

Now suppose $E_1$ is disjoint from the strands of $T$ and use $E_1$ to $\bdry$--compress $D_2$ into two disks $\Delta$ and $\Delta'$.  
Each puncture of $\Delta \cup \Delta'$ is also a puncture of $D_2$, as $E_1$ contains no punctures. As $\epsilon$ separates the punctures of $F$ and as $E_1$ is unpunctured, this implies that $\Delta$ and $\Delta'$ are each once-punctured by the strands of $T$.  As $F$ is a twice-punctured disc, each of $\bdry \Delta$ and $\bdry \Delta'$ bounds a disc in $F$ containing a single puncture of $F$. Since no strand of $T$ contains a local knot, the discs $\Delta$ and $\Delta'$ are inessential, being isotopic rel-$\bdry$ into $F$.  Since $D_2$ may be recovered by banding $\Delta$ and $\Delta'$ together along an arc in $F$ dual to $\epsilon$, it follows that $D_2$ is isotopic rel-$\bdry$ to the twice punctured disk $F$ in $\bdry T$.  However this contradicts that $D_2$ is essential. 

Thus it follows that every arc component of $D_1 \cap D_2$ is an arc separating the punctures in both $D_1$ and $D_2$.  In particular, if there are are arc components of $D_1 \cap D_2$, then both $D_1$ and $D_2$ are twice punctured, the arcs separate their punctures, and there are no simple closed curves of $D_1 \cap D_2$.

So now let $\zeta$ be an arc of $D_1 \cap D_2$ that is outermost in $D_1$, bounding a disk $E_1 \subset D_1$ that intersects $D_2$ only in $\zeta$ and is punctured once by the strands of $T$.  Use $E_1$ to $\bdry$--compress $D_2$ to disks $\Delta$ and $\Delta'$ which are each twice punctured.  Again let $\epsilon = \bdry E_1 \cut \zeta$ and let $F$ be the disk in $\bdry T$ bounded by $\bdry D_2$ that contains $\epsilon$.  As before, $\epsilon$ must separate the punctures of $F$.  However, this now implies that $\bdry \Delta$ bounds a disk in $F$ that is once-punctured by the strands of $T$.  Yet since $\Delta$ is homologous through the tangle to this disk, the strands of $T$ must also intersect $\Delta$ an odd number of times, a contradiction.

\medskip

Finally, we may return to the case that $D_1 \cap D_2$ contains loops parallel to $\bdry D_1$ and $\bdry D_2$.  However now we may assume there are no arcs of intersection.  That is, we may assume $\bdry D_1 \cap \bdry D_2 = \emptyset$.

Since the boundary of $T$ is a sphere with four punctures, if $|\bdry D_1 \cap \bdry D_2| = 0$, then $\bdry D_1$ is isotopic to $\bdry D_2$ as desired.  Then if one is disjoint from the strands of $T$ it follows that the other must be as well since they are essential disks.  Hence either both disks are disjoint from the strands of $T$ or both are intersected twice.  Since $T$ has no local knotting, if the disks are disjoint from the strands, then $T$ must be a rational tangle and so the two disks are isotopic.   

So suppose the disks both intersect the strands twice.  As we are presently assuming $|\bdry D_1 \cap \bdry D_2| = 0$, suppose $D_1 \cap D_2$ is a non-empty collection of closed curves.  First consider an innermost disk $E$ of $D_1 \cup D_2$ in the complement of the  intersection $D_1 \cap D_2$.  Because the disks are essential, $E$ must be punctured by the strands of $T$.  Because there is no local knotting in the strands of $T$, $E$ must be punctured twice.  Hence the curves of $D_1 \cap D_2$ are all parallel to $\bdry D_1$ in $D_1 - T$ and also parallel to $\bdry D_2$ in $D_2 -T$.  Observe that a curve $C$ of $D_1 \cap D_2$ outermost in $D_1$ must also be outermost in $D_2$.  So $C$ cuts off annuli $A_1$ from $D_1$ and  $A_2$ from $D_2$ where $C = A_1 \cap A_2$. Then since there are no closed components of the strands of $T$, the annulus $A_1 \cup A_2$ is boundary parallel and guides a proper isotopy of $D_1$ that reduces $|D_1 \cap D_2|$, contrary to assumption.  Hence, when isotoped to intersect minimally, $D_1 \cap D_2 = \emptyset$.
\end{proof}

The following is a standard result concerning rational tangles, although we do not know of a source for the proof. It can also be proved using branched double covers.

\begin{cor}[cf. {\cite[page 203]{KauffmanLambropoulou}}]
\label{cor:integral sum}
A marked tangle that is a tangle sum is integral if and only if each summand is integral.
\end{cor}
\begin{proof}
It follows immediately from the definition that the tangle sum of integral tangles is integral. So suppose that $T$ is an integral tangle that decomposes as the sum $T = T_1 + T_2$. 
As an integral tangle, $T$ is a rational tangle.  Thus $T$ has no local knotting and no closed components. Furthermore being rational implies it has an essential disk $D_1$ disjoint from the strands of $T$ and separating them.  As a tangle sum, $T$ contains a disk  $D_2$ that splits $T$ into the tangles $T_1$ and $T_2$, intersecting the strands of $T$ twice. Observe that by the definitions of integral tangle and tangle sum, $\bdry D_1$ and $\bdry D_2$ are not isotopic in $\bdry T$.
 Lemma~\ref{lem:lowpunct disks} therefore implies that $D_2$ cannot be an essential disk.  Thus $D_2$ is either compressible or $\bdry$--parallel.

If $D_2$ were compressible, say by a compressing disk $E$ in $T_2$, then the result of the compression would bound a disk $D_2'$ in $T_2$ that is disjoint from the strands of $T_2$ and separates them.  Hence it must also separate the strands of $T$. Thus $D_2'$ is also an essential disk for $T$.  However as $\bdry D_2'$ is isotopic to $\bdry D_2$, it is not isotopic to $\bdry D_1$.  This contradicts Lemma~\ref{lem:lowpunct disks}.

Therefore $D_2$ must be $\bdry$--parallel.
 Thus either $T_1$ or $T_2$, say $T_2$, is an integral tangle through which the parallelism occurs.  Yet then taking the tangle sum of $T$ with the mirror $\overline{T_2}$ of $T_2$ would yield $T+\overline{T_2} = T_1 + T_2 + \overline{T_2} = T_1$. (Note that $T_2 + \overline{T_2}$ is the $0$--tangle precisely because $T_2$ is an integral tangle.) As $T_1$ is now expressed as the sum of two integral tangles, it too is integral.
\end{proof}

\begin{cor}
\label{cor:unique summing slope}
Suppose the marked tangle $T$ has no local knotting or closed components and decomposes as a tangle sum along an essential summing disk. 
Then any other essential twice-punctured disk in $T$ has boundary with slope $\infty$, the slope of the boundary of the summing disk.
\end{cor}
\begin{proof}
Let $D_1$ be the disc defining the tangle sum and let $D_2$ be any other essential twice-punctured disk. Since both are twice-punctured essential disks, then by Lemma \ref{lem:lowpunct disks} their boundaries are isotopic in the complement of the marked points.
\end{proof}

\section{Proof of Theorem \ref{main theorem} and Theorem \ref{thm:propermain}}

\subsection{Case 1: When $\thetan$ is locally knotted.}

\begin{prop}\label{prop:unknottinglocalknot}
If $\thetan$ is a locally knotted $\theta$--curve with $\uprop(\thetan)=1$, then:
\begin{itemize}
    \item $\thetan = K \#_2 \thetan_0$ where $K$ is a non-trivial prime knot with $\uprop(K)=1$ and $\thetan_0$ is a trivial $\theta$--curve; and
    \item any proper rational unknotting arc for $\thetan$ can be isotoped to be disjoint from the summing sphere.
\end{itemize}
\end{prop}

\begin{proof}
Assume $\thetan$ is locally knotted.
Let $\alpha$ be a proper rational unknotting arc for $\thetan$.  Then the endpoints of $\alpha$ lie on either one or two edges of $\thetan$.  Let $e$ be an edge of $\thetan$ that does not meet $\alpha$ and let $e_1$ be another edge.  Suppose that  $\thetan$ has a local knot $K$ in the edge $e$.   Without loss of generality $\alpha$ meets the knot $C=e\cup e_1$ at most once, along the edge $e_1$.  Note that since $K$ is in $e$, the knot $C$ is nontrivial.  After a proper RTR along $\alpha$, the knot $C$ is unchanged.  Thus, $\uprop(\thetan)\neq 1$.  This is a contradiction.  So $\thetan$ has no local knotting in the edge $e$. 

\smallskip

Consider the tangle $T = (B, t)$ that is the exterior of the edge $e$.  Since $\alpha$ is disjoint from $e$, the tangle $T$ contains $\alpha$. Let $M'$ be the double branched cover of $B$, branched over $t$. The arc $\alpha$ lifts to a knot $\tilde{\alpha}$ in $M'$. By assumption, performing the proper RTR on $t$ using $\alpha$ converts $T$ into a trivial tangle. Let $M$ be the double branched cover of $B$ with branch set the strands of this trivial tangle. Observe that $M$ is a solid torus. Let $\tilde{\alpha}^*$ be the lift of the dual arc to $\alpha$; it is a knot in $M$. The proper RTR along $\alpha$ induces a non-integral Dehn surgery on $\tilde{\alpha}$ converting $M'$ to $M$ and with $\tilde{\alpha}^*$ being the core of the surgery solid torus. It is non-integral in the sense that on the torus that is the boundary of a regular neighborhood of $\tilde{\alpha}$ the slope of the meridian for $\tilde{\alpha}$ and the slope of the meridian for $\tilde{\alpha}^*$ intersect minimally at least twice. Scharlemann proved\footnote{See the comment at the end of the paper for statement (e), although there is a typo: the first $M'$ in the comment should be an $M$.}:

\begin{theorem}[Abridged main theorem of  \cite{Scharlemann:Producing}]\label{thm:scharlemann}
Let $k$ be a knot in a compact orientable $3$--manifold $M$ whose exterior is irreducible and $\bdry$--irreducible.   Let $M'$ be a manifold obtained by Dehn surgery on $k$.
%, and let $k'$ be the surgery dual knot.  
If $\bdry M$ compresses in $M$ or $M$ contains a sphere not bounding a rational homology ball, then either
\begin{enumerate}
    \item[](a),(b),(d) $M'$ is irreducible,
    \item[](c) $k$ is cabled and the slope of the surgery is that of the cabling annulus, or
    \item[](e) $M = (S^1 \times S^2) \# W$ where $W$ is a rational homology sphere.
\end{enumerate} 
\qed
\end{theorem}

Each summing sphere for $T$ lifts to an essential sphere in $M'$, so $M'$ is reducible.  Thus, neither (a), (b), nor (d) can occur. Option (c) implies that $\tilde{\alpha}$ is a cable knot and that the surgery slope is the slope of the cabling annulus. Such a slope intersects the meridian of $\tilde{\alpha}$ exactly once, rather than at least twice. Consequently, Option (c) is impossible. Option (e) is impossible since $M$ is a solid torus.

%Option (e) implies $M$ contains a non-separating sphere. By the Equivariant Sphere Theorem \cite{MSY}, the non-separating sphere may be isotoped to be either disjoint from its image under the involution or invariant under the involution meeting the fixed locus transversally.  In the former case, the sphere is disjoint from the fixed set of the involution and descends to a sphere in the ball $B$ of the tangle $T$ that is disjoint from the branch set $t$.  However as $t$ has no closed components, this sphere must bound a ball disjoint from $t$.  Hence the non-separating sphere in $M$ must also bound a ball, a contradiction.  In the latter case, the quotient of the non-separating sphere is a sphere in the ball $B$ of the tangle $T$ that meets $t$.  Such a sphere is separating and lifts to a separating sphere in $M$, a contradiction.  Hence option (e) cannot occur.

Since the conclusions of Scharlemann's Theorem~\ref{thm:scharlemann} fail, some hypothesis must be wrong.  In particular, the exterior of $K'$ (the lift of $\alpha$) must be either reducible or $\bdry$--reducible.  Since the boundary of the exterior of $K'$ is two tori, a compression of one yields a sphere that doesn't bound a ball.  Hence if the exterior of $K'$ is $\bdry$--reducible, then it is reducible as well.  Thus we assume the exterior of $K'$ is reducible.  Again by the Equivariant Sphere Theorem \cite{MSY} we may assume there is a reducing sphere $\tilde{S}$ such that either (i) $\tau(\tilde{S}) \cap \tilde{S} = \emptyset$ or (ii) $\tau(\tilde{S}) = \tilde{S}$ where $\tilde{S}$ is transverse to the fixed set of $\tau$.

In (i), $\tilde{S}$ is disjoint from the fixed set of $\tau$ and projects to a sphere $S'$ disjoint from the strands $t$ of the tangle $T$.  Then $S'$ bounds a $3$--ball $B'$.  Since $t$ has no closed components, $B' \cap t = \emptyset$.  Therefore $\tilde{S}$ bounds a ball, a contradiction.

In (ii), $\tilde{S}$ projects to a sphere $S'$ that intersects the strands $t$ of the tangle $T$ twice but is disjoint from $\alpha$.   As $S'$ must bound a $3$--ball $B'$ in $B$ and $t$ has no closed components, $B' \cap t$ is a single arc.  This arc must be non-trivial since otherwise $\tilde{S}$ will bound a ball.  Observe that this implies $S'$ is a summing sphere for $\thetan$.

Since $\alpha$ is disjoint from $S'$, either $\alpha \subset B'$ or $\alpha$ is disjoint from $B'$.  If $\alpha$ is disjoint from  $B'$, then after the proper RTR along $\alpha$, $B'$ still meets $T$ in a non-trivial knotted arc.  This is a contradiction.  Hence $\alpha \subset B'$.  Then the proper RTR along $\alpha$ must unknot the arc $B' \cap t$. This implies that $\thetan$ is the connected sum of a $\theta$-curve $\thetan_1$ and a proper rational unknotting number one knot $\kappa$. The arc $B' \cap t$ is the result of removing an arc from $\kappa$. 
The argument of \cite{Zhangsonepagepaper} that knots with unknotting number 1 are prime (cf. \cite{Scharlemann:Unknotting}) immediately generalizes to show that knots with proper rational unknotting number 1 are prime.
Since $\alpha$ is a proper rational unknotting arc for $\thetan$, the curve $\thetan_1$ must be the trivial $\theta$-curve. Consequently, $S'$ is, up to isotopy, the unique summing sphere for $\thetan$; in particular, $S$ is isotopic to $S'$, as desired.
\end{proof}

\subsection{Case 2: When $\thetan$ is a vertex sum and not locally knotted.}

Throughout this subsection we assume our $\thetan$--curves and tangles are not locally knotted.

\begin{prop}\label{prop:unknottingvertexsum}
Assume $\thetan$ is a vertex connected sum of two non-trivial theta curves.  Further assume $\thetan$ has no local knots. Then $\uprop(\thetan)>1$.
\end{prop}

\begin{proof}
For contradiction, we assume $\uprop(\thetan)=1$ and $\alpha$ is 
the arc of a proper rational tangle replacement that unknots $\thetan$. 
Then $\alpha$ will meet $\thetan$ in one or two edges.  As before, let $e$ be an edge of $\thetan$ that does not meet $\alpha$.
By hypothesis, $\thetan$ has an essential summing sphere $S$ meeting $\thetan$ in three points but there is no essential sphere meeting $\thetan$ in just two points.
For ease of exposition, we choose a fixed diagram of $\thetan$ and a fixed regular neighborhood $N(e)$ of $e$ such that near the diagram $S$ intersects the plane of projection as in Figure~\ref{1-EdgeEandS}.

\begin{figure}[ht!]
\includegraphics[width=.75\textwidth]{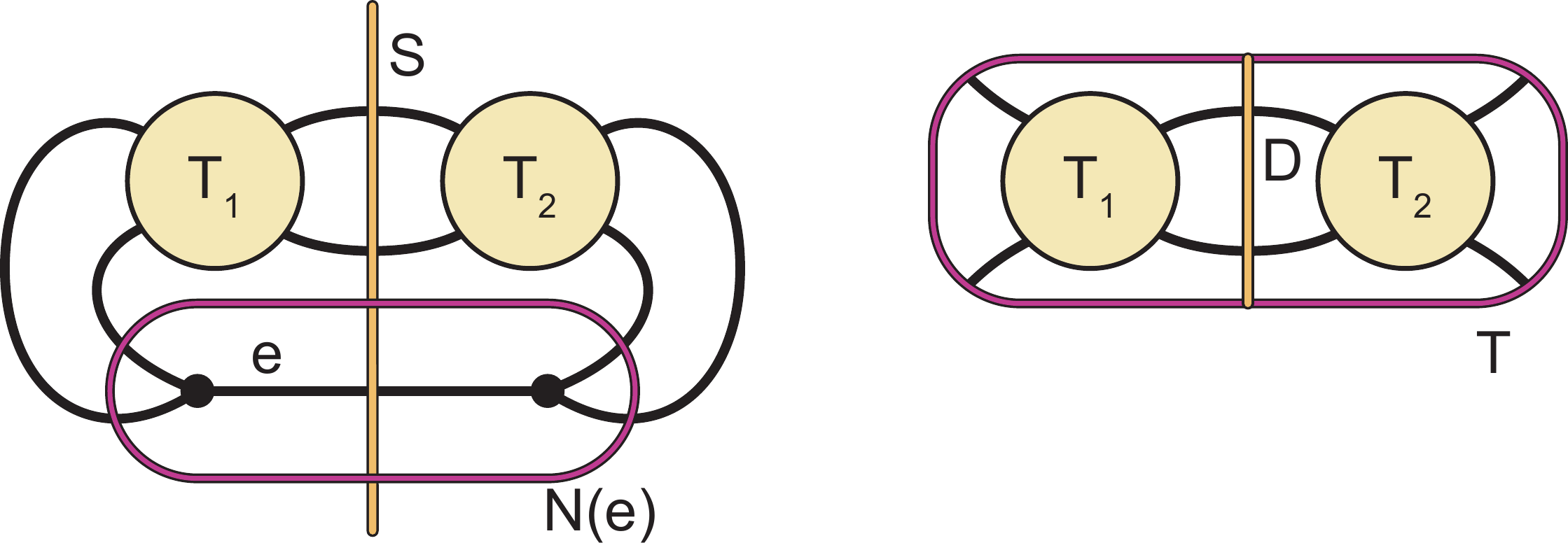}
\caption{Left: The edge $e$ and sphere $S$ for the theta curve $\thetan$. Right: The exterior of the edge $e$ is the tangle $T$.  The sphere $S$ meets $T$ in a disk $D$ expressing $T$ as the tangle sum $T_1 + T_2$. }
\label{1-EdgeEandS}
\end{figure}

Let $T$ be the tangle whose underlying 3-ball is $X(e)$, the closure of the complement of $N(e)$, and whose strands are the intersection of $\thetan$ with this 3-ball. The intersection of $S$ with $T$ is a disc $D$ intersecting the strands of $T$ exactly twice.
As in Figure \ref{1-EdgeEandS}, we may consider $T$ as a marked tangle with the disk $D$ decomposing $T$ as a tangle sum of two marked tangles $T_1$ and $T_2$, thus $T = T_1 + T_2$.

Since the sphere $S$ is essential, $D$ is an essential summing disk of $T$.  Consequently, neither $T_1$ nor $T_2$ is an integral tangle. Moreover, neither $T_1$ nor $T_2$ may be a split tangle since otherwise $\thetan$ would contain a non-trivial knot summand contrary to hypothesis.

Since $\alpha$ is 
an arc along which a proper rational tangle replacement unknots $\thetan$,
 the unknotting proper rational tangle replacement along $\alpha$ must transform $T$ into an integral tangle.  However, as $T$ is a non-trivial tangle sum, Proposition~\ref{prop:unknottingtanglesum} below shows this is impossible in the case of a crossing change, whereas Proposition~\ref{prop:properrationalunknottingtanglesum} shows this is impossible in the case of a proper rational tangle replacement of distance greater than two.
Hence we arrive at a contradiction. \end{proof}

\section{Tangle sums and proper rational tangle replacements}

In this section, we investigate proper rational tangle replacements with respect to tangle sums. 
We then construct a specific tangle $\rho$ (shown in Figure \ref{fig:rho}) whose denominator closure has sufficiently high unknotting number and proper rational unknotting number. 
This tangle will serve a specific role in the proofs of both Proposition \ref{prop:unknottingtanglesum} and Proposition \ref{prop:properrationalunknottingtanglesum}, which handle the cases of crossing changes (distance two) and proper rational tangle replacements (distance greater than two) along marked tangle sums separately. 

\begin{prop}\label{prop:essconwaydisjointsummingdisk}
Suppose that $C$ is a collection of mutually disjoint essential Conway spheres in an essential tangle $T$. If a marked tangle $T$ is the tangle sum of tangles, neither of which is integral, then there exists an essential summing disc for $T$ disjoint from $C$. 
\end{prop}

\begin{proof}
Out of all essential summing disks for $T$ that are transverse to $C$, let $D$ be chosen to minimize $|D \cap C|$.  For a contradiction, assume that $|C \cap D| >0$.

\begin{claim}\label{claim:bothpointsind}
$C$ cuts $D$ into a disk with two marked points and a collection of annuli. 
\end{claim}

\begin{proof}
A standard innermost loop argument shows that the intersection consists of some number of curves, each of which bounds an annulus in $D$ with $\partial D$ that is disjoint from the marked points. For completeness, we provide details.
Let $\gamma$ be a curve of $C \cap D$ that is innermost in $D$, and let $d \subset D$ be the subdisk that it bounds.  

Suppose $d$ is disjoint from the marked points on $D$. Then, since $C$ is essential, $\gamma$ must also bound a disk $c \subset C$ that is disjoint from the four marked points on $C$. Then $c \cup d$ is a sphere disjoint from the strands of $T$.  Since $T$ is a tangle in a ball, this sphere $c \cup d$ must bound a ball.  Because the tangle $T$ has no closed components, this ball is disjoint from the strands of $T$.  Hence, there is an isotopy of $D$ with support in a neighborhood of this ball that takes $c$ through the ball to $d$ and then just past $D$. Such an isotopy removes the intersection $\gamma$ (and possibly other intersections) without introducing new intersections.  Thus it contradicts the minimality of $|C \cap D|$.

 If $d$ contained just one of the marked points, then as every sphere in a ball separates, $\gamma$ would have to bound a disk $c \subset C$ that contains just one of the four marked points in $C$.  However then $c \cup d$ would be a sphere bounding a ball in $T$ that intersects the strands in $T$ in a single arc.  Since $T$ has no local knotting, this arc must be an unknotted arc running from $c$ to $d$.  Hence, as in the previous situation, we may isotope $d$ along the arc through $d$ and just past $c$ so remove the intersection $\gamma$ without introducing new intersections.  Thus this again contradicts the minimality of $|C \cap D|$.
 \end{proof}

\begin{claim}\label{claim:twopointsinc}
$D$ cuts $C$ into two disks each with two marked points and a collection of annuli.
\end{claim}

\begin{proof}
We will use another innermost circle argument to show that any curve of $C \cap D$ that is innermost in $C$ bounds a subdisk containing exactly two marked points of $C$. This implies that $C \cap D$ consists of a family of parallel curves that separate the four marked points of $C$ into two pairs. From that our claim follows. 

Let $\gamma$ be a curve of $C \cap D$ that is innermost in $C$, and let $c \subset C$ be the subdisk that it bounds. Out of all such curves, we may assume that we chose $\gamma$ and $c$ to contain as few marked points as possible. Since $C$ has four marked points, this means that $c$ has two or fewer marked points. We show that it has exactly two. 

Suppose $c$ is disjoint from the marked points of $c$.  Then since $\gamma =\bdry c$ cobounds an annulus $a \subset D$ with $\bdry D$ that is disjoint from the marked points of $D$ (by Claim~\ref{claim:bothpointsind}), we may form a disk $\delta = a\cup c$ with $\bdry \delta = \bdry D$ that is disjoint from the strands of $T$.  But then $T$ is a split tangle, contradicting the assumption that $T$ is essential.  

Now suppose $c$ contains only one marked point of $C$.  Then since $\gamma = \bdry c$ bounds a subdisk $d \subset D$ containing the two marked points of $D$ (by Claim~\ref{claim:bothpointsind}), we may form a sphere $c \cup d$ that meets $T$ in three points.  Since this sphere bounds a ball in the ball underlying $T$, this is again a contradiction.
\end{proof}

Let $a$ be the outermost annulus of $D \cut C$.  Then $\bdry a$ is  $\bdry D$ and a component $\gamma$ of $C \cap D$.  Let $C_\gamma$ be the component of $C$ containing $\gamma$.  Let $c'$ and $c''$ be the two complementary disks of $C_\gamma$ bounded by $\gamma$.  Then let $D' = a \cup c'$ and $D'' = a \cup c''$.  These are disks that each meet $T$ in two points and have $\bdry D = \bdry D' = \bdry D''$.  Furthermore, a slight isotopy  makes $C$ disjoint from both $D'$ and $D''$.  Thus, to contradict the minimality of $|C \cap D|$, it remains to see that at least one of $D'$ and $D''$ is an essential summing disk.  

Observe that $D'$ and $D''$ decompose $T$ into a sum $T=T_1 + T_2 + T_3$ where $C_\gamma$ bounds the tangle $T_2$.  Hence $T_2$ is an essential tangle, and in particular it is not integral. If both $D'$ and $D''$ were inessential summing disks, then both $T_1$ and $T_3$ would be integral by Corollary~\ref{cor:integral sum}. However, if both $T_1$ and $T_3$ are integral tangles, then the component $C_\gamma$  of $C$ is isotopic to the boundary of the ball.  But this contradicts that each component of $C$ is an essential Conway sphere. Thus at least one of $D'$ and $D''$ is an essential summing disk.   
\end{proof}

Let $\alpha$ continue to denote an unknotting arc for $\thetan$, or more generally, the arc of a proper rational tangle replacement that unknots $\thetan$.

\begin{lemma}\label{lem:essSummingdiskalpha}
Every essential summing disk for $T$ must intersect $\alpha$.
\end{lemma}

\begin{proof}
Suppose an essential summing disk for $T$ that is disjoint from $\alpha$ expresses $T$ as the the sum $T=T_1 + T_2$.  Thus neither $T_1$ nor $T_2$ is integral, and we may assume $\alpha$ is contained in $T_2$.   

Say performing the crossing change (or proper RTR) along $\alpha$ in $T_2$ produces $T_2^*$.  Then after the crossing change (proper RTR) along $\alpha$ in $T$, the summing disk expresses the resulting tangle as $T^* = T_1 +T_2^*$. 
Since $T^*$ is an integral tangle, Corollary~\ref{cor:integral sum} implies that both $T_1$ and $T_2^*$ are integral tangles.  However this contradicts that $T_1$ is non-integral.
\end{proof}

\begin{figure}
    \centering
    \includegraphics[width=.9\textwidth]{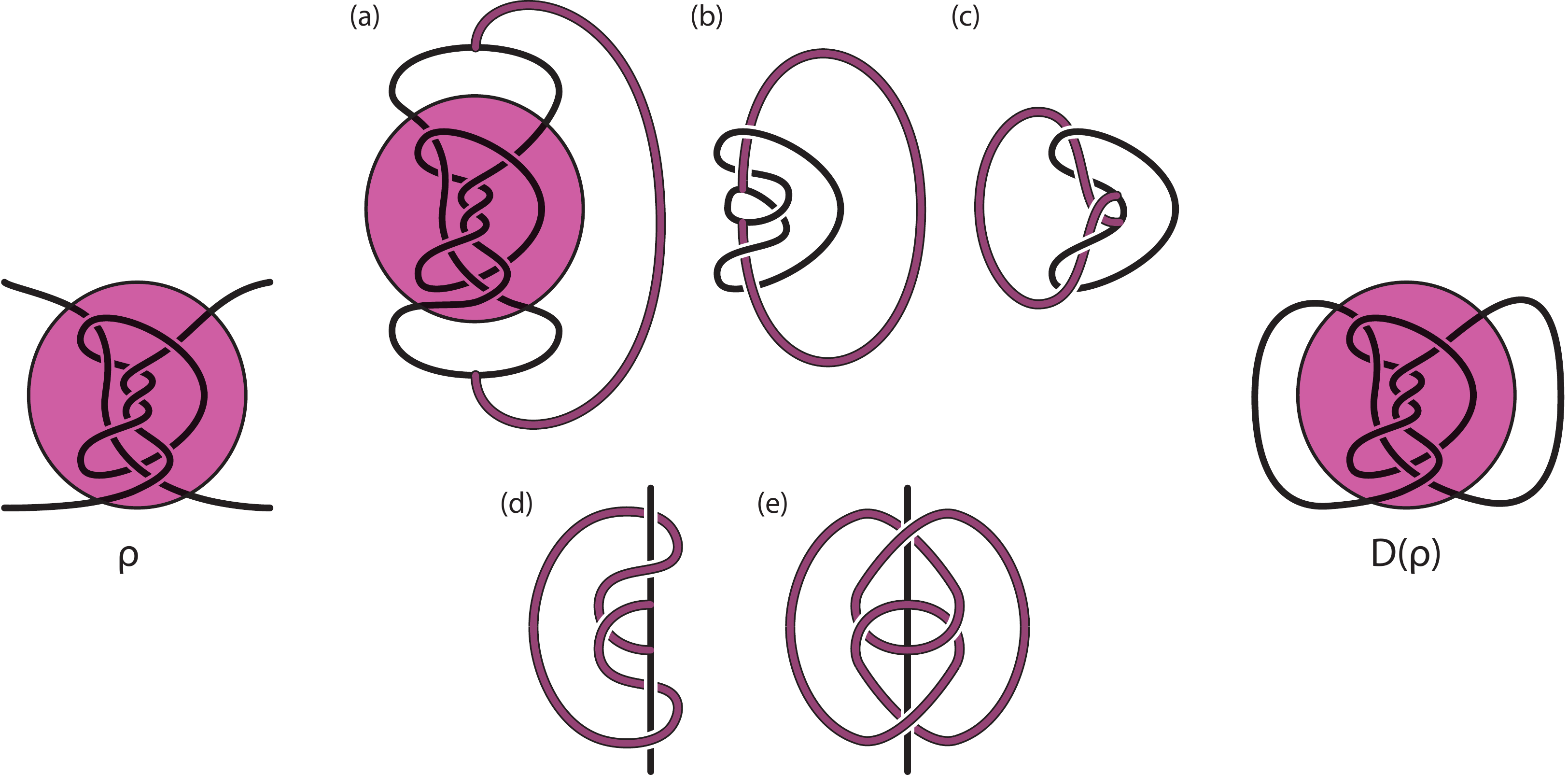}
    \caption{Left: The tangle $\rho$.  Center: The normal closure of $\rho$ with its core arc (a) is isotopic to an unknotted axis with an arc (b)--(d).  The double branched cover along the axis (e) is the figure eight knot, $4_1$. Right: The denominator closure $D(\rho)$ of $\rho$ is observed to be the knot $10_{139}$.}
    \label{fig:rho}
\end{figure}

\begin{claim}\label{claim:rho}
The tangle $\rho$ shown in Figure~\ref{fig:rho} 
has unknotted numerator closure $N(\rho)$, 
 is not a non-trivial tangle sum, and contains no essential torus, essential Conway sphere, or essential annulus. In particular, the double branched cover of $\rho$ is the exterior of the hyperbolic knot $4_1$, the figure eight knot.  Furthermore, its denominator closure is the knot $10_{139}$ which has unknotting number $4$ and proper rational unknotting number $\geq 2$.
\end{claim}

\begin{proof}
Ignoring the extra arc, Figure~\ref{fig:rho}(a)--(c) shows that $N(\rho)$ is the unknot.
This extra arc, however, is the core arc of the trivial rational tangle joined to $\rho$ to make $N(\rho)$.  Figure~\ref{fig:rho}(a)--(d) shows the straightening of the unknot $N(\rho)$ at the expense of twisting up this arc.  Figure~\ref{fig:rho}(e) shows the knot that is the lift of the arc to the double branched cover along the unknot.  One readily observes that this is the figure eight knot, $4_1$.
Since $\rho$ was the exterior tangle of the core arc, this shows that the double branched cover of $\rho$ is the exterior of the figure eight knot $4_1$.  

If $\rho$ were to contain an essential summing disk, an essential torus, an essential Conway sphere, or essential annulus then the double branched cover of $\rho$ would contain either an essential annulus or an essential torus and hence be non-hyperbolic.  However, as the double branched cover of $\rho$ is the exterior of the hyperbolic knot $4_1$, this cannot be.

The denominator closure of $\rho$, $D(\rho)$, is a 10 crossing knot which SnapPy \cite{snappy} identifies as the knot $10_{139}$ (which one may also confirm by hand).   As informed by KnotInfo \cite{knotinfo}, this knot has unknotting number $4$, see \cite{kawamura, gibsonishikawa}.  It also has proper rational unknotting number $\geq 2$ \cite{properrationalunknotting}.
\end{proof}

\subsection{Sums of non-integral tangles admit no crossing changes to integral tangles.}

In what follows, we study crossing changes on 2-strand tangles, making heavy use of work of Eudave-Mu\~noz and Gordon--Luecke. Eudave-Mu\~noz used tangles to construct a family of strongly invertible hyperbolic knots having distance 2 Dehn surgeries to toroidal 3-manifolds. Gordon and Luecke show that unknotting number one knots with essential Conway spheres have a close relationship to Eudave-Mu\~noz’s construction.

\begin{prop}\label{prop:unknottingtanglesum}
Let $T$ be a marked tangle sum of tangles, neither of which is integral. Then $T$ admits no crossing change to an integral tangle.
\end{prop}

\begin{proof}[Proof of Proposition~\ref{prop:unknottingtanglesum}]
Let the marked tangle $T$ be expressed as $T=T_1 + T_2$ for some non-integral tangles $T_1$ and $T_2$. By the definition of tangle sum, the summing disks for this decomposition of $T$ has boundary of slope $\infty$. 
Indeed, by Corollary~\ref{cor:unique summing slope}, every essential twice-punctured disk in tangle $T$ has boundary slope $\infty$.

Assume for contradiction that $T$ admits a crossing change along an arc $\alpha$ to an integral tangle.
By adding the mirror of that integral tangle to $T$, say along $T_2$, and updating $T$ to be that result, we may assume the crossing change along $\alpha$ transforms $T$ into the $0$--tangle.

Let $K_\rho = N(\rho+T)$, the numerator closure of $\rho+T$ where $\rho$ is the tangle shown in Figure~\ref{fig:rho}.
Since a crossing change along $\alpha$ transforms $T$ to the $0$--tangle,  it also transforms $K_\rho = N(\rho+T)$ into $K_\rho^* = N(\rho+0) = N(\rho)$ which is the unknot by Claim~\ref{claim:rho}.  Hence $K_\rho$ is a knot with unknotting number $1$ and unknotting arc $\alpha \subset T$.
Because both $T$ and $\rho$ are essential tangles, it follows that the sphere $\bdry \rho$ (which is isotopic to $\bdry T$) is an essential Conway sphere for $K_\rho$.  Thus $K_\rho$ is a knot with unknotting number $1$ and an essential Conway sphere.

Therefore, by \cite[Theorem 6.2]{GordonLuecke}, up to mirroring $K_\rho$, either
\begin{enumerate}
    \item any unknotting arc for $K_\rho$ may be isotoped to be disjoint from all essential Conway spheres and essential tori in the exterior of $K_\rho$,
    \item $K_\rho$ is an EM-knot, or 
    \item $K_\rho$ is the union of an essential tangle $P$ and an essential EM-tangle $P_0$. 
\end{enumerate}
Furthermore, in case (3), every unknotting arc for $K_\rho$ is isotopic into the EM-tangle $P_0$ and the standard unknotting arc for $P_0$ continues to be an unknotting (trivializing) arc for $K_\rho$.  Figure~\ref{fig:EM-tangles} shows the two families of EM-tangles, $A_1(\ell,m)$ and $A_2(\ell,m)$, with their standard unknotting arcs and essential summing disks, see \cite[Figure 4.3]{GordonLuecke}.  They are depicted, up to adding integral tangles, so that the summing disk decomposes the EM-tangles as a non-trivial tangle sum.  So that the summing disks are essential and the tangle sum is non-trivial: for $A_1(\ell,m)$ we assume $|\ell|>1$, $m\neq 0$ and $(\ell,m) \neq (2,1)$ or $(-2,-1)$; for $A_2(\ell,m)$ we assume $|\ell|>1$, $m \neq 0,1$.  Also shown are the rational tangles resulting from the crossing change along the standard unknotting arc.  In these tangles, the summing disk is unique (see eg. \cite[Addendum to Theorem 4.3]{GordonLuecke}).

\bigskip
\noindent
{\bf Case (1).}
Assume the unknotting arc $\alpha$ for $K_\rho$ may be isotoped to be disjoint from all essential Conway sphere and essential tori in the exterior of $K_\rho$.

\begin{claim}\label{claim:T1orT2essential}
At least one of $T_1$ or $T_2$ is an essential tangle.
\end{claim}

\begin{proof}
If not, then both are rational tangles since we are assuming $T$ is not locally knotted.  Since neither $T_1$ nor $T_2$ is an integral tangle, then $D(T)$, the denominator closure of $T$, is a connected sum of two non-trivial two-bridge knots.  Then the crossing change along $\alpha$ taking $T$ to the $0$--tangle transforms $D(T)$ to the unknot.  Yet this contradicts that unknotting number one knots are prime \cite{Scharlemann:Unknotting}.
\end{proof}

Say $T_2$ is an essential tangle.  Then since $\rho+T_1$ is not split (as it is neither rational nor locally knotted),  $\bdry T_2$ is an essential Conway sphere in $K_\rho$.  After a slight push into the interior of $T$ we may assume $\bdry T_2$ is disjoint from $\rho$.  Then, by the hypothesis of this case, we may assume that $\alpha$ is disjoint from both $\bdry \rho$ and $\bdry T_2$.

\smallskip
If $\alpha$ were contained in $T_2$, then the crossing change would change $T_2$ into $T_2^*$.  Thus $T_1+ T_2^*$ must be the $0$--tangle since that is the only way to fill $\rho$ to get the unknot.  However the $0$--tangle, indeed any integral tangle, only decomposes as a sum of two tangles if the two tangles are integral by Corollary~\ref{cor:integral sum}.  Yet this implies that $T_1$ is an integral tangle, contrary to assumption.  Hence $\alpha$ is not contained in $T_2$.

\smallskip
If $\alpha$ were contained in $T$ but not in $T_2$, 
then the crossing change does not alter that each of the four strands of the tangle between $\rho$ and $T_2$ connect $\bdry \rho$ to $\bdry T_2$.  Hence the crossing change cannot make either of these Conway spheres become inessential.  

Indeed, suppose there were a compressing disk $D$ for one of the spheres after the crossing change.  Isotop it rel--$\bdry$ to intersect the two spheres minimally. Being a compressing disk for one of the spheres, any intersections in the interior of $D$ must be with the other sphere.  If there are any such intersections, then they will cut off an innermost subdisk $D'$ from $D$.  By the presumed minimality, $D'$ must be a compressing disk too.  Hence we may restrict attention to a compressing disk $D$ that meets the two spheres only in its boundary.  Such a disk $D$ would have to be contained in the $S^2 \times I$ tangle between $\rho$ and $T_2$ since both those tangles are essential. Yet then $D$ must separate a strand of the tangle from the other strands, preventing it from connecting $\bdry \rho$ to $\bdry T_2$.  This is a contradiction.

Since the result of the crossing change on $K_\rho$ has an essential Conway sphere, it cannot be the unknot.  Hence $\alpha$ is not contained in $T$. 

\smallskip
Thus $\alpha$ may be isotoped into $\rho$.  Suppose the crossing change transforms $\rho$ into $\rho^*$ as it changes $K_\rho$ to $K_\rho^*$.   By assumption, $K_\rho^*$ is the unknot.
If $\rho^*$ is not a rational tangle, then again $\bdry \rho^*$ will be an essential Conway sphere in $K_\rho^*$ which should be the unknot, a contradiction. Hence $\rho^*$ must be rational.  
Because $\rho$ is homotopic to the $1$--tangle and a crossing change will not alter the homotopy type, $\rho^*$ is also homotopic to the $1$--tangle.  In particular, $\rho^*$ is a rational tangle that is not the $\infty$--tangle. 

Suppose $T_1$ is an essential tangle.  Since $\rho^*$ is not the $\infty$--tangle, $\rho^*+T_1$ must again be an essential tangle.  Then, as $T_2$ is also an essential tangle,  $\bdry (\rho^* + T_1) = \bdry T_2$ is an essential Conway sphere for $K_\rho^*$, a contradiction.
Hence $T_1$ must be an inessential tangle.  By assumption it is not locally knotted and non-integral, so $T_1$ must be a non-integral rational tangle.

If $\rho^*$ is also a non-integral rational tangle, then $\rho^* + T_1$ is a tangle sum of two non-integral (and non-$\infty$) rational tangles.  Hence $\rho^* + T_1$ is an essential tangle.  Therefore, as before, $\bdry (\rho^* + T_1) = \bdry T_2$ is an essential Conway sphere for $K_\rho^*$, a contradiction.

Therefore $\rho^*$ must be an integral tangle.  Hence $D(\rho^*)$, the denominator closure of $\rho^*$, is an unknot.   However, since $D(\rho^*)$ is obtained from the knot $D(\rho)$, the denominator closure of $\rho$, by a crossing change along $\alpha$, then $D(\rho)$ has unknotting number $1$.  This contradicts that $D(\rho)$ has unknotting number $2$ as determined in Claim~\ref{claim:rho}.

\medskip

\noindent
{\bf Case (2).}
An EM-knot has a unique Conway sphere by \cite{EudaveMunoz} (see also \cite[Theorem 6.2(2)(b)]{GordonLuecke}).  This sphere splits the EM-knot into two Montesinos tangles, each a tangle sum of two non-integral rational tangles.   While the essential Conway sphere $\bdry \rho$ splits $K_\rho$ into the tangles $T$ and $\rho$, the tangle $\rho$ is not a tangle sum as shown in Claim~\ref{claim:rho}.  Hence $K_\rho$ cannot be an EM-knot as in case (2) above.

\medskip

\noindent
{\bf Case (3).}
Thus $K_\rho$ must be as in case (3) above: there is an EM-tangle $P_0$ contained in $K_\rho$ so that $\bdry P_0$ is an essential Conway sphere and every unknotting arc for $K_\rho$ may be isotoped into $P_0$.     Furthermore, the standard unknotting arc for $P_0$ must continue to be an unknotting arc for $K_\rho$.  Hence we shift out attention upon the standard unknotting arc of $P_0$, regardless of whether $\alpha$ is isotopic to it in $K_\rho$.

\begin{claim}
Either $P_0 \subset T$ where $\bdry P_0$  is an essential Conway sphere or $P_0 =T$.
\end{claim}

\begin{proof}
Assume the Conway spheres $\bdry \rho$ and $\bdry P_0$ have been isotoped in $K_\rho$ to intersect minimally.  Since both $\bdry \rho$ and $\bdry P_0$ are essential Conway spheres in $K_\rho$, the curves of $\bdry \rho \cap \bdry P_0$ must all be isotopic in each $\bdry \rho$ and $\bdry P_0$ and bound disks in both that intersect the strands twice.  In particular, the components of $\rho \cap \bdry P_0$ must either be essential twice-punctured disks or essential annuli. However by Claim~\ref{claim:rho}, $\rho$ may contain neither.  Hence $\bdry \rho \cap \bdry P_0 = \emptyset$.

If $\bdry \rho$ is contained in $P_0$, then since $\bdry \rho$ is essential in $K_\rho$, either $\bdry \rho$ is also essential in $P_0$ or it cobounds a product tangle with $\bdry P_0$. The latter implies that we may take $\bdry \rho=\bdry P_0$ by an isotopy through the product tangle.  Since the double branched cover of an EM-tangle such as $P_0$ is a graph manifold (see \cite[p2069--2070]{GordonLuecke}), neither is it a hyperbolic manifold nor does it contain any essential torus bounding a hyperbolic manifold.  Since the double branched cover of $\rho$ is a hyperbolic manifold by Claim~\ref{claim:rho}, $\bdry \rho$ may only be contained in $P_0$ if $\bdry \rho = \rho \cap P_0$.
If $\bdry P_0$ is contained in $\rho$, then since $\bdry P_0$ is essential in $K_\rho$, we similarly conclude that either $\bdry P_0$ is also essential in $\rho$ or $\bdry P_0 = \bdry \rho$ after an isotopy. However since $\rho$ contains no essential Conway spheres by Claim~\ref{claim:rho}, this implies $\bdry P_0 = \bdry \rho$.  Since we have already ruled out the possibility that $P_0 = \rho$, we conclude that  $P_0 = T$ if $\bdry P_0 \subset \rho$.
%Since $P_0$ contains an essential Conway sphere (see \cite[p2069]{GordonLuecke}) while $\rho$ does not, $P_0$ cannot be contained in $\rho$.  this implies $\rho \cap P_0 = \bdry \rho$.    
Thus $P_0$ must be contained within $T$.

Finally, observe that if $\bdry P_0$ is not an essential Conway sphere in $T$, the either $\bdry P$ is isotopic to $\bdry T$ through a product tangle in $T$ or there is a compressing disk $d$ for $\bdry P_0$ in $T$ disjoint from the strands.  In the former, we may then assume $P_0=T$.  In the latter,  the disk $d$ would need to become trivial upon inclusion of $T$ into $K_\rho = N(\rho+T)$ since $\bdry P_0$ is essential in $K_\rho$; yet this cannot occur since $\bdry d$ does not bound a disk in $\bdry P_0$ disjoint from the strands.  Hence either $\bdry P_0$ is an essential Conway sphere in $T$ or $P_0 = T$.
\end{proof}

If $P_0$ is contained in $T$ where $\bdry P_0$ is an essential Conway sphere, then by Proposition~\ref{prop:essconwaydisjointsummingdisk} there is an essential summing disk $D_0$ for $T$ that is disjoint from $P_0$.  Assume $D_0$ decomposes $T$ into the sum $T_1+T_2$ and, say, that $P_0 \subset T_2$.  While $\alpha$ must intersect $D_0$ as an arc in $T$ by Lemma~\ref{lem:essSummingdiskalpha}, there must be an isotopy of $\alpha$ through $K_\rho$ to bring it to an arc $\alpha_0$ in $P_0 \subset T_2$ since we are in case (2). The crossing change on $\alpha$ then becomes a crossing change on $\alpha_0$ taking $T_2$ to a tangle $T_2'$.  Thus the crossing changes on $\alpha_0$ takes $K_\rho = N(\rho + T_1 + T_2)$ to $N(\rho + T_1 + T_2')$ which should be an unknot.  However as $\rho$ is a non-trivial tangle whose numerator closure is the unknot by Claim~\ref{claim:rho}, we must have $T_1 + T_2'$ be the $0$--tangle.  Yet this requires both $T_1$ and $T_2'$ to be integral tangles by Corollary~\ref{cor:integral sum}.  As $D_0$ is an essential summing disk, $T_1$ is not an integral tangle, a contradiction. 

If $P_0 = T$ then the summing disk of $P_0$ must be a summing disk of $T$. 
%(The EM-tangles have unique summing disks by Lemma~\ref{lem:EMtangles}.)  
Since $T$ is homotopic to the $0$--tangle, so must be $P_0$.  Thus $P_0$ cannot be an EM-tangle $A_1(\ell,m)$ since they are homotopic to tangles disjoint from their summing disks and hence to the $\infty$--tangle. (See Figure~\ref{fig:EM-tangles}.)  Thus $P_0$ must be an EM-tangle $A_2(\ell,m)$.  Since $T$ admits a crossing change to the $0$--tangle, $P_0$ must also admit a crossing change to the $0$--tangle.  In particular, the standard unknotting arc for $P_0$ must change $P_0$ to the $0$--tangle.  
However, the unknotting crossing change along the standard unknotting arc for an EM-tangle $A_2(\ell, m)$ produces the rational tangle $[0, -\ell, -m, -2] = \frac{2m-1}{\ell(2m-1)-2}$.
This is integral only if $\ell(2m-1)-2 = \pm 1$ which implies that either $\ell = \pm1$ or $m=0,1$.  Yet these choices for $\ell$ or $m$ make the summing disk of $P_0$ inessential so that $P_0$ is not a non-trivial tangle sum.  This is a contradiction. 
\end{proof}

\begin{figure}
    \centering
    \includegraphics[width=\textwidth]{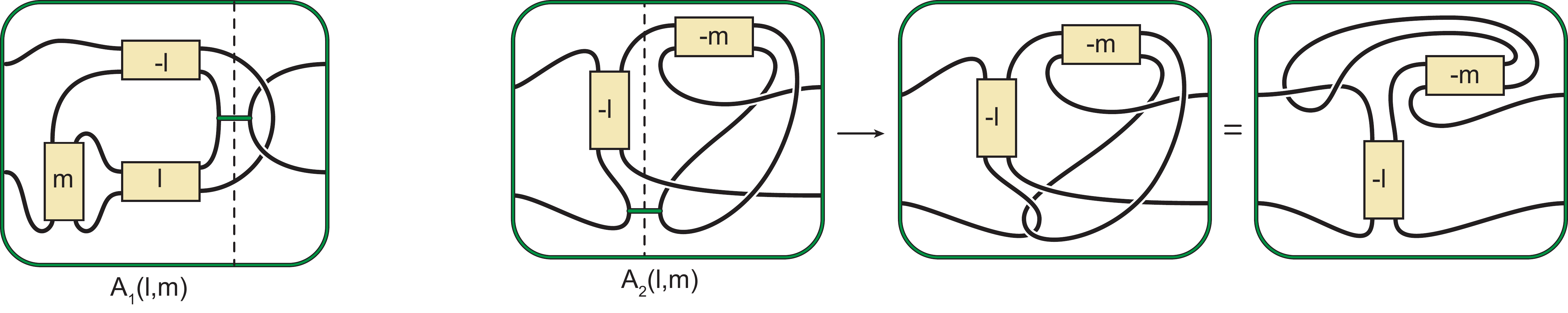}
    \caption{Left: EM-tangle $A_1(\ell,m)$ with its summing disk and standard unknotting arc.  Right: EM-tangle $A_2(\ell,m)$ with its summing disk and standard unknotting arc, followed by the unknotting crossing change along this arc, and an isotopy to more visibly be a rational tangle. (EM-tangels are shown up to adding integral tangles.)}
    \label{fig:EM-tangles}
\end{figure}

\subsection{Sums of non-integral tangles admit no proper RTRs to integral tangles.}

\begin{prop}\label{prop:properrationalunknottingtanglesum}
Let $T$ be a marked tangle sum of tangles, neither of which is integral. Then $T$ admits no proper rational tangle replacement to an integral tangle.
\end{prop}

\begin{proof}
This proof follows similarly to that of Proposition~\ref{prop:unknottingtanglesum} which proves the result for crossing changes.  Since proper rational tangle replacements have distance $\Delta \geq 2$ and those of distance $\Delta =2$ are realized as crossing changes, we may restrict attention to proper rational tangle replacements of distance $\Delta >2$.

As in the proof of Proposition~\ref{prop:unknottingtanglesum}, assume for contradiction that a tangle sum $T=T_1 + T_2$ admits a proper rational tangle replacement (of distance $\Delta >2$) along an arc $\alpha$ to the $0$--tangle.  Then, again using the tangle $\rho$ from Figure~\ref{fig:rho}, we may form the knot $K_\rho = N(\rho+T)$ which the proper rational tangle replacement transforms into the unknot $K_\rho^* - N(\rho+0)$. Let $\alpha^*$ be the arc on $K_\rho^*$ dual to $\alpha$ (ie.\ the core arc of the replacing rational tangle).  Hence $K_\rho$ is a knot with proper rational unknotting number $1$ along the arc $\alpha \subset T$.  As before, $\bdry \rho$ is an essential Conway sphere for $K_\rho$.

Now we shall show that $\alpha$ cannot be isotoped to be disjoint from all essential Conway spheres for $K_\rho$.  This argument follows exactly the same as {\bf Case (1)} of Proposition~\ref{prop:unknottingtanglesum} for crossing changes using that 
\begin{itemize}
    \item a proper rational tangle replacement also preserves the connectivity of the strands (ie.\ it preserves the homotopy type of the rational tangle), and 
    \item $D(\rho)$ actually has proper rational unknotting number $\geq 2$ by Claim~\ref{claim:rho} instead of just the regular unknotting number $\geq 2$.
\end{itemize}

Now, instead of \cite[Theorem 6.2]{GordonLuecke}, we must apply \cite[Theorem 5.2]{GordonLuecke} to the double branched cover of $X=K_\rho \cut \nbhd(\alpha) = K_\rho^* \cut \nbhd(\alpha^*)$, the exterior of the arc $\alpha$ on the knot $K_\rho$, which is also the exterior of the arc $\alpha^*$ on the unknot $K_\rho^*$.   In particular, since $K_\rho^*$ is the unknot, $\alpha^*$ lifts to a knot $k$ in $S^3$ whose exterior is $X=S^3 \cut \nbhd(k)$. In the context of \cite[Theorem 5.2]{GordonLuecke}, $S^3 = X(\mu)$ and the double branched cover of $K_\rho$ is $X(\gamma)$. Then the distance $\Delta(\mu,\gamma)$ between the curves $\mu$ and $\gamma$ in $\bdry X$ is the distance between the Dehn fillings, this agrees with the distance $\Delta$ of the rational tangle replacement. Therefore, since $\Delta >2$, only cases (1) and (4) of   \cite[Theorem 5.2]{GordonLuecke} apply.  However, in each of these two cases, the knot that is the lift of $\alpha$ is disjoint from the essential tori of $X(\gamma)$.  Since we have shown that $\alpha$ cannot be isotoped to be disjoint from all of the essential Conway spheres of $K_\rho$, its lift cannot be isotoped to be disjoint from all of the essential tori of $X(\gamma)$ that are the lifts of these essential Conway spheres.  This contradicts \cite[Theorem 5.2]{GordonLuecke}, and thus there can be no proper rational tangle replacement from $K_\rho$ to the unknot $K_\rho^*$.  Consequently, $T$ admits no proper rational tangle replacement to an integral tangle.
\end{proof}

\section*{Acknowledgements}
 
This paper was initiated as part of the SQuaRE program at the American Institute of Mathematics (AIM). We thank AIM for their support and hospitality.
KLB was partially supported by a grant from the Simons Foundation (grant \#523883  to Kenneth L.\ Baker).
DB acknowledges generous funding from The Leverhulme Trust (Grant RP2013-K-017) and additional support from EPSRC grant EP/H0313671. 
AHM is supported by The Thomas F. and Kate Miller Jeffress Memorial Trust, Bank
of America, Trustee. 
SAT is partially supported by NSF Grant DMS-2104022 and a Colby College Research Grant.

%---------------------------------------------------------------------
%   BIBLIO
%---------------------------------------------------------------------

\bibliographystyle{alpha}
\bibliography{bibliography}

\newcommand{\etalchar}[1]{$^{#1}$}
\begin{thebibliography}{SdlCMR{\etalchar{+}}98}

\bibitem[ASZ{\etalchar{+}}92]{AdamsCozzarelli}
David~E. Adams, Eugene~M. Shekhtman, E.~Lynn Zechiedrich, Molly~B. Schmid, and
  Nicholas~R. Cozzarelli.
\newblock The role of topoisomerase {IV} in partitioning bacterial replicons
  and the structure of catenated intermediates in {DNA} replication.
\newblock {\em Cell}, 71(2):277--288, 10 1992.

\bibitem[BG21]{Barbensi}
Agnese Barbensi and Dimos Goundaroulis.
\newblock {$f$}-distance of knotoids and protein structure.
\newblock {\em Proc. A.}, 477(2246):Paper No. 20200898, 18, 2021.

\bibitem[CDGW]{snappy}
Marc Culler, Nathan~M. Dunfield, Matthias Goerner, and Jeffrey~R. Weeks.
\newblock Snap{P}y, a computer program for studying the geometry and topology
  of $3$-manifolds.
\newblock Available at \url{http://snappy.computop.org} (05/11/2021).

\bibitem[Con70]{conway}
J.~H. Conway.
\newblock An enumeration of knots and links, and some of their algebraic
  properties.
\newblock In {\em Computational {P}roblems in {A}bstract {A}lgebra ({P}roc.
  {C}onf., {O}xford, 1967)}, pages 329--358. Pergamon, Oxford, 1970.

\bibitem[EMn97]{EudaveMunoz}
Mario Eudave-Mu\~{n}oz.
\newblock Non-hyperbolic manifolds obtained by {D}ehn surgery on hyperbolic
  knots.
\newblock In {\em Geometric topology ({A}thens, {GA}, 1993)}, volume~2 of {\em
  AMS/IP Stud. Adv. Math.}, pages 35--61. Amer. Math. Soc., Providence, RI,
  1997.

\bibitem[GDBS17]{Goundaroulis2}
Dimos Goundaroulis, Julien Dorier, Fabrizio Benedetti, and Andrzej Stasiak.
\newblock Studies of global and local entanglements of individual protein
  chains using the concept of knotoids.
\newblock {\em Scientific Reports}, 7(1):6309, 2017.

\bibitem[GGL{\etalchar{+}}17]{Goundaroulis1}
Dimos Goundaroulis, Neslihan Gügümcü, Sofia Lambropoulou, Julien Dorier,
  Andrzej Stasiak, and Louis Kauffman.
\newblock Topological models for open-knotted protein chains using the concepts
  of knotoids and bonded knotoids.
\newblock {\em Polymers}, 9(9), 2017.

\bibitem[GI02]{gibsonishikawa}
William Gibson and Masaharu Ishikawa.
\newblock Links and {G}ordian numbers associated with generic immersions of
  intervals.
\newblock {\em Topology Appl.}, 123(3):609--636, 2002.

\bibitem[GL06]{GordonLuecke}
C.~McA. Gordon and John Luecke.
\newblock Knots with unknotting number 1 and essential {C}onway spheres.
\newblock {\em Algebr. Geom. Topol.}, 6:2051--2116, 2006.

\bibitem[ILM21]{properrationalunknotting}
Damian Iltgen, Lukas Lewark, and Laura Marino.
\newblock Khovanov homology and rational unknotting, 2021.

\bibitem[Kaw98]{kawamura}
Tomomi Kawamura.
\newblock The unknotting numbers of {$10_{139}$} and {$10_{152}$} are {$4$}.
\newblock {\em Osaka J. Math.}, 35(3):539--546, 1998.

\bibitem[KL04]{KauffmanLambropoulou}
Louis~H. Kauffman and Sofia Lambropoulou.
\newblock On the classification of rational tangles.
\newblock {\em Adv. in Appl. Math.}, 33(2):199--237, 2004.

\bibitem[Lac21]{Lackenby}
Marc Lackenby.
\newblock Links with splitting number one.
\newblock {\em Geom. Dedicata}, 214:319--351, 2021.

\bibitem[LM21]{knotinfo}
Charles Livingston and Allison~H. Moore.
\newblock Knotinfo: Table of knot invariants.
\newblock URL: \url{knotinfo.math.indiana.edu}, November 2021.

\bibitem[LMRH{\etalchar{+}}12]{LopezSchvartzman}
V.~Lopez, M.~L. Martinez-Robles, P.~Hernandez, D.~B. Krimer, and J.~B.
  Schvartzman.
\newblock {{T}opo {I}{V} is the topoisomerase that knots and unknots sister
  duplexes during {D}{N}{A} replication}.
\newblock {\em Nucleic Acids Res.}, 40(8):3563--3573, Apr 2012.

\bibitem[MSY82]{MSY}
William Meeks, III, Leon Simon, and Shing~Tung Yau.
\newblock Embedded minimal surfaces, exotic spheres, and manifolds with
  positive {R}icci curvature.
\newblock {\em Ann. of Math. (2)}, 116(3):621--659, 1982.

\bibitem[MT11]{MT}
Sergei Matveev and Vladimir Turaev.
\newblock A semigroup of theta-curves in 3-manifolds.
\newblock {\em Mosc. Math. J.}, 11(4):805--814, 822, 2011.

\bibitem[Sch85]{Scharlemann:Unknotting}
Martin~G. Scharlemann.
\newblock Unknotting number one knots are prime.
\newblock {\em Invent. Math.}, 82(1):37--55, 1985.

\bibitem[Sch90]{Scharlemann:Producing}
Martin Scharlemann.
\newblock Producing reducible {$3$}-manifolds by surgery on a knot.
\newblock {\em Topology}, 29(4):481--500, 1990.

\bibitem[SdlCMR{\etalchar{+}}98]{SantamariaSchvartzman}
D.~Santamaria, G.~de~la Cueva, M.~L. Martinez-Robles, D.~B. Krimer,
  P.~Hernandez, and J.~B. Schvartzman.
\newblock {{D}na{B} helicase is unable to dissociate {R}{N}{A}-{D}{N}{A}
  hybrids. {I}ts implication in the polar pausing of replication forks at
  {C}ol{E}1 origins}.
\newblock {\em J. Biol. Chem.}, 273(50):33386--33396, Dec 1998.

\bibitem[Tur12]{Turaev}
Vladimir Turaev.
\newblock Knotoids.
\newblock {\em Osaka J. Math.}, 49(1):195--223, 2012.

\bibitem[VHK{\etalchar{+}}96]{VigueraSchvartzman}
E.~Viguera, P.~Hernandez, D.~B. Krimer, A.~S. Boistov, R.~Lurz, J.~C. Alonso,
  and J.~B. Schvartzman.
\newblock {{T}he {C}ol{E}1 unidirectional origin acts as a polar replication
  fork pausing site}.
\newblock {\em J. Biol. Chem.}, 271(37):22414--22421, Sep 1996.

\bibitem[Wol87]{Wolcott}
Keith Wolcott.
\newblock The knotting of theta curves and other graphs in {$S^3$}.
\newblock In {\em Geometry and topology ({A}thens, {G}a., 1985)}, volume 105 of
  {\em Lecture Notes in Pure and Appl. Math.}, pages 325--346. Dekker, New
  York, 1987.

\bibitem[Zha91]{Zhangsonepagepaper}
Xingru Zhang.
\newblock Unknotting number one knots are prime: a new proof.
\newblock {\em Proc. Amer. Math. Soc.}, 113(2):611--612, 1991.

\end{thebibliography}
%\bibliography{UN1Theta}

\end{document}